\documentclass[10pt]{amsart}

\usepackage{amsmath, amsfonts, latexsym,amsthm, amsxtra, amssymb,graphicx,cite}

\newtheorem{theorem}{Theorem}[section]

\newtheorem{proposition}{Proposition}[section]
\newtheorem{corollary}{Corollary}[section]

\newtheorem{lemma}{Lemma}[section]
\newtheorem{remark}{Remark}[section]

\def\CC{{\mathbb C}}
\def\ZZ{{\mathbb Z}}

\def\HH{{\mathcal H}}
\def\GG{{\mathcal G}}

\def\cC{{\mathcal C}}
\def\cU{{\mathcal U}}
\def\cK{\mathcal K}
\def\cL{{\mathcal L}}
\def\cS{{\mathcal S}}
\def\cD{\mathcal D}

\title[On a family of symmetric hypergeometric functions]{On a family of symmetric hypergeometric functions of several variables and their Euler type integral representation}

\author{Zhuangchu LUO}
\address{School of Mathematics and Statistics, Wuhan
University, Wuhan 430072, China}
\email{zhchluo.math@whu.edu.cn}

\author{Hua CHEN}
\address{School of Mathematics and Statistics, Wuhan
University, Wuhan 430072, China}
\email{chenhua@whu.edu.cn}

\author{Changgui ZHANG}
\address{ Laboratoire P. Painlev\'e (UMR -- CNRS 8524), UFR
Math., Universit\'e de Lille 1, Cit\'e scientifique, 59655
Villeneuve d'Ascq cedex, France}
\email{zhang@math.univ-lille1.fr}


\date{December 19, 2011}
\begin{document}

\maketitle

\begin{abstract}
This paper is devoted to the family $\{G_n\}$ of hypergeometric series of any finite number of variables, the coefficients being the square of the multinomial coefficients $(\ell_1+...+\ell_n)!/(\ell_1!...\ell_n!)$, where $n\in\ZZ_{\ge 1}$. All these series belong to the family of the general Appell-Lauricella's series. It is shown that each function $G_n$ can be expressed by an integral involving the previous one, $G_{n-1}$. Thus this family can be represented by a multidimensional Euler type integral, what suggests some explicit link with the Gelfand-Kapranov-Zelevinsky's theory of $A$-hypergeometric systems or with the Aomoto's theory of hypermeotric functions. The quasi-invariance of each function $G_n$ with regard to the action of a finite number of involutions of $\CC^{*n}$ is also established. Finally, a particular attention is reserved to the study of the functions $G_2$ and $G_3$, each of which is proved to be algebraic or to be expressed by the Legendre's elliptic function of the first kind.
\end{abstract}


\numberwithin{equation}{section}
\section{Introduction}\label{section:introduction}

The subject of the present paper is the family of germs of analytic functions $G_n$ at $0\in\CC^n$ given by the following symmetric power series:
\begin{equation}\label{equation:Gnseries}
G_n(x)=\sum_{\ell\in{\mathbb N}^n}\bigl(\frac{\vert\ell\vert !}{\ell !}\bigr)^2\,x^\ell\quad (n=1,2,3,4,...)\,;
\end{equation}
for convention, we set $G_0\equiv1$. Here, the following notations are used: for any positive integer $n$, if $x=(x_1,...,x_n)\in\CC^n$ and $
\ell=(\ell_1,..,\ell_n)\in{\mathbb N}^n$, then:
\begin{equation}\label{equation:ell}
 x^\ell={x_1}^{\ell_1}\,...\,{x_n}^{\ell_n}\,,\quad
 \vert\ell\vert=\ell_1+...+\ell_n,\quad
\ell!=\ell_1!...\ell_n!\,.
\end{equation}

As it will be explained later in \S\ref{subsection:relatedtopics}, these series belong to the family of Appell-Lauricella's series and, in the same time, they may be also considered as examples of Gelfand-Kapranov-Zelevinsky's $A$-hypergeometric functions or as being in relation with Aomoto's theory of hypergeometric functions. Some motivations and future works related to the present paper will be set out in the straight line of these explanations.

We will start by expounding some results established in the present work. The general plan of the paper will be indicated in \S\ref{subsection:plan}, at the end of this introduction.

\subsection{Main results}\label{subsection:mainresults}
One of our main results can be stated as follows.

\begin{theorem}\label{theorem:Gnintegral} The following general formula holds for any positive integer $n$ and any $x=(x',x_n)$ close enough to $0\in\CC^n$:
\begin{equation}\label{equation:Gnintegral}
G_n(x)=\frac{1}{2\pi i}\int_0^{(1+)}G_{n-1}\bigl(\frac{t\,x'}{(t-1)(1-x_nt)}\bigr)\,\frac{dt}{(t-1)(1-x_nt)},
\end{equation}
where the integration contour is composed of a closed path starting from the origin and making a circle around $t=1$ in the anti-clock sense.
\end{theorem}

To each non-zero complex number $a$, we associate the quadratic transformation $\cS(a;t)$ of $\CC_\infty=\CC\cup\{\infty\}$ as follows:
\begin{equation}\label{equation:Sa}
\cS(a;t)=\frac{t}{(t-1)(1-at)}\,;
\end{equation}
more generally, for any integer $m\ge 2$ and any vector $\vec a=(a_1,...,a_m)\in{\CC^*}^m$, we define the following transformation of $\CC_\infty^m$:
\begin{equation}\label{equation:Sam}
\cS(a_1,...,a_m;t_1,...,t_m)=\cS\bigl(\cS(a_m,t_m)\,(a_1 ,a_2,...,a_{m-1});(t_1,t_2,...,t_{m-1})\bigr)\,.
\end{equation}

In addition, it will be convenient to introduce the following notations. For  $1\le j\le n$ and $x=(x_1,...,x_n)$, we write
\begin{equation}\label{equation:xj'}
\bar x_j'=(x_1,...,x_j),\quad
\bar x_j=(x_j,...,x_n),
\end{equation}
so that $x=(\bar x_{j-1}',\bar x_j)$.

Let be denoted by $\cC_1$ any smooth closed-loop homotopic in $\CC\setminus\{1\}$ to the positive-oriented circle $\partial^+D(1;R)$ centered at the point of unity affix  and possessing some radius $R\in(0,1)$.
One can deduce from Theorem \ref{theorem:Gnintegral} the following statement.
\begin{theorem}\label{theorem:GncC}The following relation holds for any positive integer $n$ and all $x$ sufficiently near to $0$ in $\CC^n$:
\begin{equation}\label{equation:GnS}
G_n(x)=\frac {1}{(2\pi i)^n}\,\int_{\cC_1^n}\prod_{j=1}^n \cS(\bar x_j;\bar t_j)\,\frac{dt_j}{t_j}\,,
\end{equation}
where $\bar x_j$ and $\bar t_j$ are as given in \eqref{equation:xj'} for $x=(x_1,...,x_n)$ and $t=(t_1,...,t_n)$, respectively.
\end{theorem}

\bigskip
By making use of the notations \eqref{equation:xj'}, we define the family of involutions $T_{n,j}$ of ${\CC^*}^n$ as follows:
\begin{equation}\label{equation:Tj}
T_{n,j}(x)=(\bar x'_{j-1}/x_j,1/x_j,\bar x_{j+1}/x_j)\,,
 \end{equation}
where $x=(\bar x'_{j-1},x_j,\bar x_{j+1})$. Let $\GG_n$ be the sub-group of ${\rm Aut}({\CC^*}^n)$ generated by all the transformations $T_{n,j}$. One can easily check that all permutations of the coordinates $(x_1,...,x_n)$ belong to $\GG_n$.
 If $F$ denotes any function given in a domain $\Omega\subset{\CC^*}^n$ such that $\sigma\Omega=\Omega$ for all $\sigma\in\GG_n$,  the action of $\GG_n$ on $F$ is induced by the relation $\sigma F(x)=F(\sigma x)$ for $\sigma\in\GG_n$.

Let
\begin{equation}\label{equation:HnQn}
H_n(x)=\sqrt{Q_n(x)}\,G_n(x)\,,\quad
Q_n(x)=\bigl(1-\sum_{j=1}^nx_j\bigr)^2-2\sum_{i\not=j}x_i\,x_j\,.
\end{equation}
Obviously, $Q_n$ is a quadratic symmetric polynomial in $x$; with respect to the action of $T_{n,j}$, one can easily find that
\begin{equation}\label{equation:TjQn}
T_{n,j}\,Q_n(x)=\frac1{x_j^2}\,Q_n(x)\,.
\end{equation}

\begin{theorem}\label{theorem:Hnsymmetry}
For any positive integer $n$, the above-defined function $H_n(x)$ is left invariant by the action of $\GG_n$.

More explicitly, it follows that $
H_{1}=H_2=1
$
and that $H_3$ is related with the Lengendre's elliptic function of the first kind in the following manner:
\begin{equation}\label{equation:H3}
H_3(x)={}_2F_1\bigl(\frac14,\frac34;1;u(x)\bigr)\,,\quad
u(x)=\frac{64\,x_1\,x_2\,x_3}{Q_3(x)^2}\,.
\end{equation}
\end{theorem}

Remark that the {\it new} variable $u(x)$ appeared in \eqref{equation:H3} is invariant under the action of any element of the group $\GG_3$.

\subsection{Preliminary commentaries}\label{subsection:relatedtopics}
The present work is situated at the meeting point of the theory of Appell-Lauricella's hypergeometric functions, the Gelfand-Kapranov-Zelevinsky's theory of $A$-hypergeometric functions, and the  Aomoto's theory of hypergeometric functions.

Firstly, it is easy to see that the series $G_n$ belong to the family of the so-called Appell-Lauricella's series. Remember that the hypergeometric functions of two variables were firstly introduced by P. Appell \cite{Ap1} in 1880   and are called currently as Appell's hypergeometric functions, that are denoted as $F_1$, $F_2$, $F_3$ and $F_4$, An extension of these functions is given in 1893 by G. Lauricella \cite{La} in the case of more than two variables and this gives raise to the functions $F_A$, $F_B$, $F_C$ and $F_D$. The starting point is to study how to extend the Gauss hypergeometric functions to the case of several complex variables.

Secondly, relation \eqref{equation:GnS} shows that each $G_n$ is really an $A$-hypergeo\-me\-tric function in the sense of Gelfand-Kapranov-Zelevinsky \cite{GKZ}. Remember that the last theory contains the classical Appell-Lauricella's functions and many mathematicians have worked for the links existing between these theories and also for their extensions or interpretations; see for example \cite{Ad, Dw, Ki,Be}, ....

Lastly, one finds a complete theory of hypergeometric functions mainly developed by Aomoto; see the book \cite{AoK}. Our integral representation \eqref{equation:GnS} can be formally transformed into an Aomoto type formula as well as given in \cite[Theorem 4]{Ki} or \cite[p. 101-102, $F_4$]{DL} but this one is not valid for our series. See in the below Remark \ref{remark:G23G2} for the case of $n=2$.

At the same time, the present work is motivated by the singularities analysis of PDEs; see  \cite{CT,CLT,LCZ} and the references therein. Indeed, we are interested in studying the confluent hypergeometric series such as
$$\sum_{n\ge0}n!(x+ y)^n,\quad
 \sum_{\ell,m\ge 0}\ell!m!x^\ell y^m\quad
  \hbox{or}\quad
   \sum_{\ell,m\ge 0}(\ell+m)!x^\ell y^m ,\quad\hbox {\it etc};
    $$all these series can be viewed as Euler type series of double variables.

    Consider for example the first one: it satisfies formally any of the following two  PDEs:
    $$
\partial_xf(x,y)=\partial_yf(x,y),\quad
(x+y)^2\partial_x f(x,y)-(1-(x+y))f(x,y)=-1;$$
the common initial datum,
$
f(0,y)=\sum_{n\ge 0}n! y^n$, is merely the Euler series \cite{Ra} which  is really a divergent power series in $\CC^*$. Expand $\sum_{n\ge0}n!(x+ y)^n$ as power series of both variables $x$ and $y$, and apply separately the Borel transform \cite{Ra} to each of these variables; then one finds the Appell's series $
F_4(1,1,1,1,x, y)$ \cite[p. 14]{AK}, that is exactly the second series of the family $G_n$ given at the beginning of the paper.

As it is announced in Theorem \ref{theorem:Hnsymmetry}, $G_2$ can be expressed as $1/\sqrt{Q_2(x_1,x_2)}$, what is a symmetric quadratic algebraic function of two variables. Such result is useful to do a Stokes analysis for the singularities of the associated PDE.

Moreover, in \cite{Ap3}, Appell outlined a survey on the Fourier-Bessel's type confluent hypergeometric functions. As the asymptotic behavior of such functions is intimately related to the singularity nature at infinity, the global monodromy problem would be lead to a Stokes analysis with several variables. See \cite{Ra1, Zh, LR} for some cases with only one variable.

In addition, in \cite{Ka1,Ka2,Ka3}, M. Kato gave a very detailed study on Appell's $F_4$,   for the connection problem between the origin and the infinity and also the monodromy group of the associated Pfaff's system. On the other hand,  the monodromy group is considered for $A$-hypergeometric systems in \cite{Be}. It seems that our present work would be helpful to doing  a similar study for a family of functional equations.

Finally, it would be interesting to find some {\it new variables} to express the functions $H_n$ in the spirit of \eqref{equation:H3}. We think this would be a good occasion to make use of classical books \cite{AK,We}, {\it etc.}; in our opinion, some projective geometry ideas would also be helpful, as in the books \cite{AoK,Yo2}.

\subsection{The plan of the paper}\label{subsection:plan}
Section \ref{section:AL} is devoted to studying the family $\{G_n\}$ and especially the functions $G_2$ and $G_3$ from a point view of Appell-Lauricella's series $F_C$. In \S\ref{subsection:ALGn}, it is shown that each $G_n$ can be viewed as a particular example of the series $F_4$ for $n=2$ and $F_C$ for $n\ge3$ variables. An algebraic expression related to the generating function of the Legendre polynomials is then deduced in \S\ref{subsection:ALG2} from some formulas that Appell established for his series $F_{1}$, ..., $F_4$. Being expressed as power series of one variable whose coefficients are given in terms of $F_4$, the function $G_3$ is proved to be quasi-invariant with respect to the action of the involutions $T_{3,j}$ of $\CC^{*3}$, $1\le j\le 3$; see Theorem \ref{theorem:G3symmetry}.

Theorems \ref{theorem:Gnintegral}  and \ref{theorem:GncC} are proved in Section \ref{section:GnI} and the proof of the first one is based upon the fact that each $G_n$ can be considered as a solution of some PDE admitting $G_{n-1}$ as the initial condition.  In \S\ref{subsection:GnIoperator}, an infinite order partial differential operator involving Euler vector fields is associated to a contour integral; this operator permits to obtain $G_n$ from $G_{n-1}$, as shown in Proposition \ref{proposition:GnGn-1}. In \S\ref{subsection:GnIS}, Theorem \ref{theorem:GncC} is proved  by making use of Theorem \ref{theorem:Gnintegral}; with help of the same theorem, the analytic continuation domain of $G_n$ will be considered in \S\ref{subsection:GnIAC}.

Section \ref{section:Sym} contains three paragraphes, the last one of which is devoted to explaining how to establish the invariance of $H_n$ with respect to the action of $\GG_n$. For doing this, we will set out two approaches for studying the action of $\GG_n$ on $G_n$: one is based upon the Euler vector fields (see \S\ref{subsection:SymEuler}) and the other upon the integral representation (see \S\ref{subsection:Symintegral}).

In order to finish the proof of Theorem \ref{theorem:Hnsymmetry} in  Section \ref{section:G23}, we consider the possible explicit expressions for $G_2$ and $G_3$ with the help of Theorem \ref{theorem:Gnintegral}. It will be shown that $G_3$ can be expressed by an elliptic integral; see Theorem \ref{theorem:G3} of \S\ref{subsection:G23G3L}. Paragraph \ref{subsection:G23G3} contains a relatively long preparation in order to put the related differential form into a Riemann's normal form.

Finally, we end the paper with a point of view of ODE about the function $G_3$, in Section \ref{section:G3ODE}. Indeed, one may find one new variable, the variable $u(x)$ as given in Theorem \ref{theorem:Hnsymmetry}, for transforming a second order PDE satisfied by $G_3$ into a Riccati equation. See Theorem \ref{theorem:riccati}.

\section{The point of view of Appell-Lauricella's  series $F_4$ and $F_C$}\label{section:AL}

We will start with the definition of the series $F_4$ and $F_C$, that allows us to see each $G_n$ as a particular case of these classical functions. In particular, from Appell's results, we get a closed expression for $G_2$ in \S\ref{subsection:ALG2} and some symmetry properties for $G_3$ in \S\ref{subsection:ALG3}.

\subsection{Writing $G_n$ by means of $F_4$ and $F_C$}\label{subsection:ALGn}
Let $(\alpha, \beta,\gamma,\gamma')\in\CC^4$ such that $\gamma\notin\ZZ_{\le 0}$ and $\gamma'\notin\ZZ_{\le 0}$. Following \cite{Ap1} and \cite[p.14]{AK}, the associated Appell's series $F_4(\alpha,\beta,\gamma,\gamma',x,y)$ is defined in the following manner:
\begin{equation}\label{eq:dfnF4}
 F_4(\alpha,\beta,\gamma,\gamma',x,y)=\sum_{\ell,m\ge 0}\frac{(\alpha)_{\ell+m}\, (\beta)_{\ell+m}}{(\gamma)_\ell\, (\gamma')_m}\frac{x^\ell y^m}{\ell!\,m!}
\end{equation}
if $(x,y)\in \Omega:=\Omega_2((0,0);1)$, where
\begin{equation}\label{eq:Omega}
\Omega_2((0,0);1)=\{(x,y)\in\CC^2: \sqrt{\vert x\vert}+\sqrt{\vert y\vert}< 1\}\,.
\end{equation}
Its link with the Euler-Gauss' hypergeometric function ${}_2F_1$ may be expressed by the following relation:
\begin{equation}\label{equation:F4Gauss}
F_4(\alpha,\beta,\gamma,\gamma',x,y)=\sum_{\ell\ge 0}\frac{(\alpha)_\ell\,(\beta)_\ell}{(\gamma)_\ell}\,{}_2F_1(\alpha+\ell,\beta+\ell;\gamma';y)\,\frac{x^\ell}{\ell!}\,.
\end{equation}
Using Barne's contour integral for ${}_2F_1$ gives raise to the expression
\begin{eqnarray*}
&&F_4(\alpha,\beta,\gamma,\gamma',x,y)=\frac{\Gamma(\gamma)\,\Gamma(\gamma')}{\Gamma(\alpha)\,\Gamma(\beta)}\,(\frac1{2\pi i})^2\cr
&&\qquad\qquad\int_{-\infty i}^{+\infty i}\int_{-\infty i}^{+\infty i}\frac{\Gamma(\alpha+s+t)\Gamma(\beta+s+t)}{\Gamma(\gamma+s)\,\Gamma(\gamma'+t)}\,\Gamma(-s)\Gamma(-t)(-x)^s(-y)^t\,dsdt\,,
\end{eqnarray*}
what allows us to consider $F_4$ on the domain $(\CC\setminus[0,+\infty[)^2$ in $\CC^2$.

In \cite{La}, Lauricella extended the four types of Appell's series $F_{1-4}$ into $F_{A-D}$ for $n$ variables. By means of the notations $x^\ell$, $\vert\ell\vert$ and $\ell!$ introduced in \eqref{equation:ell}, the definition of $F_C$ may be given as follows:
\begin{equation}\label{Apell.Fc}
F_C(\alpha,\beta,\vec \gamma,x)=\sum_{\ell\in{\mathbb N}^n}\frac{(\alpha)_{\vert\ell\vert}\,(\beta)_{\vert\ell\vert}}
{(\vec \gamma)_\ell\,\ell!}x^\ell\,,
\end{equation}
where $\vec \gamma=(\gamma_1,...,\gamma_n)\in (\CC\setminus\ZZ_{\le 0})^n$ and $(\vec \gamma)_\ell=(\gamma_1)_{\ell_1}...(\gamma_n)_{\ell_n}$.
Instead of $\Omega$ as given in \eqref{eq:Omega},  the domain of convergence of $F_C$ is
 \begin{equation}\label{equation:Omegan}
\Omega_n({\bf 0};1)= \{x\in \CC^n: \sqrt{\vert x_1\vert}+\sqrt{\vert x_2\vert}+\dots+\sqrt{\vert x_n\vert}<1\}\,;
\end{equation}
 see \cite[p. 115]{AK}.

Putting $\alpha=\beta=\gamma_1=\gamma_2=\dots =\gamma_n =1$ in $F_C$ gives raise to the family of power series $G_n$, that is to say:
\begin{equation}\label{equation:GnFC}
G_n(x)=F_C(1,1,\vec 1,x),\qquad
\vec1=(1,1,..,1)\in\CC^n\,.
\end{equation}

\subsection{An algebraic expression  for $G_2$}\label{subsection:ALG2}
Let $x=(x_1,x_2)$ and let $F_2$ be the second Appell's series:
$$
F_2(\alpha,\vec\beta,\vec\gamma,x)=\sum_{\ell\in{\mathbb N}^2}\frac{(\alpha)_{\vert\ell\vert}\,(\vec\beta)_{\ell}}{(\vec\gamma)_\ell\,\ell!}\,x^\ell\quad(\alpha\in\CC,\ \vec\beta\in\CC^2,\ \vec\gamma\in(\CC\setminus\ZZ_{\le 0})^2)\,.
$$
In the direct line of the Pfaff's transformations known for the classic hypergeometric functions, the following relation is established in \cite[p. 27]{AK}:
\begin{eqnarray}\label{equation:F4F2}
&&F_4\bigl(\alpha,\alpha+\frac 12-\beta,\gamma,\beta+\frac 12,x_1,x_2\bigr)\cr
&&\qquad\qquad=
\bigl(1+\sqrt{x_2}\bigr)^{-2\alpha}F_2\,\bigl(\alpha,(\alpha+\frac 12-\beta,\beta),(\alpha,2\beta),
\frac {x_1}{(1+\sqrt{x_1})^2},\frac {4\sqrt{x_1}}{(1+\sqrt{x_2})^2}\bigr)\,.
\end{eqnarray}

Moreover, let $F_1$ be the first Appell's series:
$$
F_1(\alpha,\vec\beta,\gamma,x)=\sum_{\ell\in{\mathbb N}^2}\frac{(\alpha)_{\vert\ell\vert}\,(\vec\beta)_\ell}{(\gamma)_{\vert\ell\vert}\,\ell!}x^\ell\quad
(\alpha\in\CC,\ \vec\beta\in\CC^2,\ \gamma\in\CC\setminus\ZZ_{\le 0});
$$
it follows \cite[p. 35, (10)]{AK} that
\begin{equation}\label{equation:F2F1}
F_2(\alpha,(\beta,\beta'),(\gamma,\alpha),x_1,x_2)=(1-x_2)^{-\beta'}
F_1\bigl(\beta,(\alpha-\beta',\beta'),\gamma,x_1,\frac {x_1}{1-x_2}\bigr)\,.
\end{equation}

Letting $\alpha=\gamma=2\beta=1$ in \eqref{equation:F4F2}  gives raise to the following relation ($x=(x_1,x_2)$):
$$
G_2(x)=(1+\sqrt{ x_2})^{-2}\,
F_2 \bigl(1,(1,\frac 12),(1,1),\frac{ x_1}{(1+\sqrt{ x_2})^2},\frac{4\sqrt{ x_2}}{(1+\sqrt{ x_2})^2}\bigr)\,,
$$
which, by considering \eqref{equation:F2F1}, can be written as follows:
\begin{eqnarray}\label{equation:G2F1}
&&G_2(x)=(1+\sqrt{ x_2})^{-2}\,\cr
&&\qquad\qquad
\bigl(1-\frac{4\sqrt{x_2}}{(1+\sqrt{x_2})^2}\bigr)^{-\frac 12}\,F_1\bigl(1,(\frac 12,\frac 12),1,\frac {x_1}{(1+\sqrt{x_2})^2}, \frac{x_1}{(1-\sqrt{x_2})^2}\bigr)\,.
\end{eqnarray}

\begin{theorem}\label{theorem:G2}Let $H_2$ and $Q_2$ be as given in \eqref{equation:HnQn} for $n=2$. Then the following relation holds for all $x=(x_1.x_2)$ near zero:
\begin{equation}\label{equation:G2}
G_2(x_1,x_2)=\frac 1{\sqrt{Q_2(x_1,x_2)}}\,.
\end{equation}
In other words, it follows that $H_2\equiv1$.
\end{theorem}

\begin{proof}
By taking into account the fact that
\[
F_1\left(1, (\frac 12,\frac 12),1,x,y\right)
=\left( (1-x)(1-y)\right)^{-\frac 12},
\]
one gets finally \eqref{equation:G2} from \eqref{equation:G2F1} by direct calculations.
\end{proof}

\begin{remark}\label{remark:G2}
If one introduces the following notations:
$$
z=\frac{x_1+x_2}{x_1-x_2},\quad
r=x_1-x_2,
$$
then the above relation \eqref{equation:G2} can be read as follows:
$$
G_2(x_1,x_2)=\frac1{\sqrt{1-2zr+r^2}}\,,
$$
that is merely the generating function of the Legendre polynomials $P_n(z)$. Thus one can find \eqref{equation:G2} by making use of the relation $P_n(z)={}_2F_1(-n,1;1;(1-z)/2)$ ({\it cf} \cite[p. 295, (6.3.5)]{AAR}).
\end{remark}

\begin{remark}\label{remark:G2extension}
From the algebraic expression \eqref{equation:G2}, it follows that $G_2$ may be extended on the universal covering $\tilde \Omega_2$ of $\CC_2^2\setminus S$ if we denote by $\CC_2$ the Riemann surface of $\sqrt z$ and by $S$ the following subset of $\CC_2^2$:
$$
S=S_+\cup S_-,\quad
S_\pm=\{(x_1,x_2)\in\CC_2\times\CC_2: \sqrt {x_1}\pm\sqrt{x_2}=1\}\,.
$$
\end{remark}

Let $\pi$ the canonical map of $\CC_2$ onto $\CC^*$ and  $\Pi=(\pi_1,\pi_2)$ the induced map from $\tilde\Omega_2$ to $\CC^2$. Let $\Omega_2=\Pi(\tilde\Omega_2)$. One may suppose that $\Omega_2$ is invariant under the action of $\GG_2$; see (1.2) for the definition of $\GG_n$.

\begin{remark}\label{remark:G2symmetry}
Theorem \ref{theorem:Hnsymmetry} is proved in the case of $n=2$.
\end{remark}

\subsection{Action of the group $\GG_3$ on $G_3$}\label{subsection:ALG3}
We start with the following  expression of $G_3$ in terms of the Appell's series $F_4$:
\begin{eqnarray}\label{equation:G3F4}
G_3(x)=\sum_{k=0}^\infty F_4(k+1,k+1,1,1,x_1,x_2)\,{x_3^k}\quad(x=(x_1,x_2,x_3))\,.
\end{eqnarray}

\begin{lemma}\label{lemma:F4}For any integer $k$, the following identity holds for all $(x,y)\in(\CC\setminus(0,+\infty))^2$ such that $\frac xy\notin (0,+\infty)$:
\begin{equation}\label{equation:F4k}
F_4(k,k,1,1,x,y)=- y^{-k}F_4\left(k,k,1,1,\frac xy,\frac 1y\right).
\end{equation}
\end{lemma}

\begin{proof}
 Recall the following relation \cite[p. 26, (37)]{AK}, which can be obtained from the connection formula of ${}_2F_1$ between zero and infinity:
\begin{eqnarray*}
&&F_4(\alpha,\beta,\gamma,\gamma',x,y)=\frac{\Gamma (\gamma ')\Gamma (\beta-\alpha)}{\Gamma (\gamma '-\alpha)\Gamma(\beta)}
(-y)^{-\alpha}\, F_4\bigl(\alpha,\alpha+1-\gamma',\gamma,\alpha+1-\beta,\frac xy,\frac 1y\bigr)\cr
 &&\qquad\qquad\qquad
 \qquad+\frac{\Gamma (\gamma ')\Gamma (\alpha-\beta)}{\Gamma(\alpha)\Gamma (\gamma '-\beta)}
(-y)^{-\beta} F_4\bigl(\beta+1-\gamma',\beta,\gamma,\beta+1-\alpha,\frac xy,\frac 1y\bigr).
\end{eqnarray*}
If we let $\alpha=k+\epsilon$, $\beta=k-\epsilon$, $\gamma=\gamma'=1$, where $k\in\ZZ$ and $\epsilon\to 0$, then the coefficients will contain $\Gamma(\pm2\epsilon)$, $\Gamma(k\pm\epsilon)$ and $\Gamma(1-k\pm\epsilon)$; a direct computation allows us to obtain the desired relation \eqref{equation:F4k}.
\end{proof}

\begin{theorem}\label{theorem:G3symmetry} The function $G_3(x)$ satisfies the following relation:
\begin{equation}\label{equation:G3Tj}
T_{3,j}G_3(x)=-{x_j}\,G_3(x)\,,\quad 1\le j\le 3.
\end{equation}
Equivalently, the function $H_3$ given in \eqref{equation:HnQn} with $n=3$ is invariant under the action of $\GG_3$.
\end{theorem}

\begin{proof}
It comes immediately from \eqref{equation:G3F4} and \eqref{equation:F4k}, by noticing that both $G_n(x)$ and $ F_4(\alpha,\alpha,\gamma,\gamma,x,y)$ are symmetrical with respect to their respective variables.
\end{proof}

\begin{remark}\label{remark:G3}
Theorem \ref{theorem:G3symmetry} may be viewed as a particular case $n=3$ of Theorem \ref{theorem:GnTn}, which extends the equality \eqref{equation:G3Tj} to the case of all positive integer $n$; see Section \ref{section:Sym}.
\end{remark}
Consequently, in order to complete the proof of Theorem \ref{theorem:Hnsymmetry} for $n=3$, it remains only to find the expression of $H_3$ by means of the Legendre's elliptic function.

\section{Proof of Theorems \ref{theorem:Gnintegral} and \ref{theorem:GncC}}\label{section:GnI}

In the following, the main goal is to give a proof for Theorem \ref{theorem:Gnintegral}. For doing this, we shall make use of the fact that $G_n$ can be viewed as solution to a Cauchy problem for which the initial condition is exactly $G_{n-1}$. This idea leads us to introduce an infinite order partial differential operator made up with  Euler-type vector fields. As in the Cauchy's theory for analytic functions, this differential operator can be represented by an integral;  see Proposition \ref{proposition:KuD} in \S\ref{subsection:GnIoperator}. In paragraphes \ref{subsection:GnIn-1} and \ref{subsection:GnIproof} we then  apply this representation to our couple of functions $(G_{n-1},G_n)$.

In \S\ref{subsection:GnIS}, we shall set out how to read Theorem \ref{theorem:GncC} as a corollary of Theorem \ref{theorem:Gnintegral}. We conclude this section by considering the analytic continuation domain of each $G_n$, by mean of Theorem \ref{theorem:Gnintegral}.

In whole this section, $n$ denotes an integer $\ge 2$,  $x=(x_1,x_2,...,x_n)$, $x'=\bar x_{n-1}'=(x_1,...,x_{n-1})$ and the following Euler operator and vector fileds will also be used:
\begin{equation}\label{equation:DD'}
\delta_{x_j}=x_j\,\frac{\partial\ }{\partial x_j},\quad
\cD=\delta_{x_1}+\dots+\delta_{x_n},\quad
\cD'=\cD-\delta_{x_n}.
\end{equation}

\subsection{An infinite order differential operator}\label{subsection:GnIoperator}
To any positive integer $k$, we associate the polynomial $P_k$ of degree $2k$ as follows:
\begin{equation}\label{equation:Ckj}
P_k(z)=\prod_{j=1}^k(z+j)^2=\sum_{j=0}^{2k}C_{k,j}z^j\in\ZZ_{\ge0}\,\,[z]\,.
\end{equation}
For convenience, we set $C_{k,j}=0$ if $j>2k$. Let $\cK(u,z)$ denote the following {\it modified generating function} associated to the double sequence $C_{k,j}$:
\begin{equation}\label{equation:Kuz}
\cK(u,z)=\sum_{k,j\ge 0}\frac{C_{k,j}}{k!^2}\,u^k\,z^j\,,
\end{equation}
that can be read as follows:
\begin{equation}\label{equation:K2F1}
\cK(u,z)=\sum_{k\ge0}\frac{P_k(z)}{k!^2}\,u^k={}_2F_1(z+1,z+1;1;u)\,.
\end{equation}
Thus, it follows that $\cK(u,z)$ can be extended into a holomorphic function in the domain $\bigl(\CC\setminus[1,+\infty)\bigr)\times\CC$ of $\CC^2$.

\begin{lemma}\label{lemma:Kuz}
For any compact $K\subset\CC\setminus[1,+\infty)$, there exist $C_K>0$ and $M_K>0$ making the following relation valid:
\begin{equation}\label{equation:Kuzestimates}
\vert \cK(u,z)\vert<C_K\,e^{M_K\,\vert z\vert},,\quad\forall (u,z)\in K\times\CC\,.
\end{equation}
\end{lemma}

\begin{proof}
From Kummer's formula \cite[p.67, (3)]{Lu} and \eqref{equation:K2F1}, it follows that
$$
\cK(u,z)=(1-u)^{-z-1}\,{}_2F_1\bigl(z+1,-z;1;\frac u{u-1}\,\bigr)\,,
$$
to what the Watson's asymptotic formula \cite[p.237, (8)]{Lu} can be applied. In the same time, one can also make use of the expression \eqref{equation:KuzS} given latter in \S\ref{subsection:GnIproof} for $\cK(u,z)$. The details are left to the reader.
\end{proof}

In other words, the function $\cK(u,z)$ admits at most the first order exponential growth at $z=\infty$ when $u$ remains inside any relatively compact subset of the cut-plan $\CC\setminus[1,+\infty)$. Therefore, one can define its Laplace transform as follows:
\begin{equation}\label{equation:Cut}
\cC(u,\zeta)=\int_0^\infty \cK(u,z)\,e^{-z\zeta}\,dz\,,
\end{equation}
and this represents an analytic function at $(0,\infty)\in\CC\times\bigl(\CC\cup\{\infty\}\bigr)$.

\begin{proposition}\label{proposition:KuD}Let $\cD'$ be as in \eqref{equation:DD'} and let $\cK(u,\cD')$ be the infinite order differential operator obtained by replacing $z$ by $\cD'$ in \eqref{equation:Kuz}. Let ${\HH}_{n-1}$ and $\HH_n$ be  the space of germs of holomorphic functions at $x'=0\in\CC^{n-1}$ and at $x=(x',x_n)=0\in\CC^n$, respectively.
Then $\cK(x_n,\cD')$ is well-defined from $\HH_{n-1}$ to $\HH_n$  and, moreover,  for all $f\in\HH_{n-1}$, one can find $R>0$ such that the following integral representation holds near $x=0\in\CC^n$:
\begin{equation}\label{equation:KuDf}
\cK(x_n,\cD')f(x)=\frac1{2\pi i}\,\int_{\cL_R}f(e^\zeta\,x')\,\cC(x_n,\zeta)\,d\zeta\,,
\end{equation}
where $\cL_R$ denotes the closed-loop along the circle $\vert z\vert=R$ in the anti-clock sense.
\end{proposition}

\begin{proof}Let $f\in\HH_{n-1}$, replace $(u,z)$ by $(x_n, \cD')$ in \eqref{equation:Kuz} and consider the obtained double power series that involves  $\cD'^jf(x')$. By induction on $j$, one may easily prove that
 $$\cD'^j f(x')= \bigl.\frac{\partial^j f(e^\zeta x')}{\partial \zeta^j} \bigr|_{\zeta=0}\,, \quad j\in\ZZ_{\ge 0}\,;$$
 therefore, from the Cauchy's formula, one finds that
$$
\cK(x_n,\cD')f(x')=\sum_{k,j\ge 0}\frac{C_{k,j}}{k!^2}\,x_n^k\,\frac{j!}{2\pi i}\,\int_{\cL_R}f(e^\zeta\,x')\,\frac{d\zeta}{\zeta^{j+1}}\,.
$$
Here, if we choose $R$ large  enough, then one can observe that the following equality holds for all $\zeta\in\CC$ satisfying $\vert \zeta\vert> R/2$:
$$
\sum_{k,j\ge 0}\frac{C_{k,j}\,j!}{k!^2}\,\frac{x_n^k}{\zeta^{j+1}}=\cC(x_n,\zeta
)\,,
$$
which allows us to permute $\sum$ and $\int$. Thus one  obtains the wanted expression \eqref{equation:KuDf}.
\end{proof}

\subsection{Passing from $G_{n-1}$ to $G_n$}\label{subsection:GnIn-1}
It is easy to check that the power series $G_n$ defined in \eqref{equation:Gnseries} satisfies the following partial differential equation:
\begin{equation}\label{Gn.PDE}
 \delta_{x_n}^2 G_n(x)=x_n\,(\cD+1)^2\, G_n(x),\quad
 G_n(x',0)=G_{n-1}(x')\,.
\end{equation}
 Notice that one can also get this PDE by the one established for $F_C$ in \cite[p. 117]{AK}, making all the parameters equal to $1$.

\begin{proposition}\label{proposition:GnGn-1} Consider the differential operator $\cK(x_n,\cD')$ defined in Proposition \ref{proposition:KuD}. Then it follows that
\begin{equation}\label{equation:GnK}
G_n(x',x_n)=\cK(x_n,\cD')G_{n-1}(x')
\end{equation}
for all $x=(x',x_0)$ close enough to $0\in\CC^n$.
\end{proposition}

\begin{proof}
From \eqref{Gn.PDE} one can deduce the following relation for any positive integer $k$:
\begin{equation}\label{Gnk.PDE}
 \bigl(\frac{1}{x_n}\delta_{x_n}^2\bigr)^k\, G_n(x)=\prod_{m=1}^k(\cD+m)^2\, G_n(x)\,.
\end{equation}
Indeed, let $j$ be an integer $\in[1,n]$;  from the expression \eqref{equation:Gnseries}, it follows that
\begin{equation}\label{equation:Gndelta}
\bigl(\frac{1}{x_j}\delta_{x_j}^2\bigr)^k G_n(x)=\sum_{\ell\in{\mathbb N}^n}\bigl(\frac{(\vert\ell\vert+k)!}
{\ell!}\bigr)^2\, x^{\ell}\,;
\end{equation}
thus one finds \eqref{Gnk.PDE} by direct computation, noticing the elementary relation $\cD x^\ell=\vert\ell\vert\,x^\ell$.
On can complete the proof of \eqref{equation:GnK} by rewriting the power series $G_n$ as follows:
\begin{equation}\label{Gn.n-1.formal}
G_n(x',x_n)=\sum_{k=0}^\infty  \sum_{\ell\in{\mathbb N}^{n-1}}\bigl(\frac{(\vert\ell\vert+k)!}
{\ell!}\bigr)^2\, x'^\ell \,\frac{x_n^k}{k!^2}\,.
\end{equation}
We omit the details.
\end{proof}

\begin{remark}\label{remark:GnGn-1}
From \eqref{Gn.n-1.formal}, one may notice the following alternative representations of $G_n$ in terms of $G_{n-1}$:
$$
G_n(x',x_n)=\sum_{k=0}^\infty \bigl(\frac{1}{x_j}\delta_{x_j}^2\bigr)^k\, G_{n-1}(x') \,\frac{x_n^k}{k!^2}=\sum_{k=0}^\infty \prod_{m=1}^k(\cD'+m)^2\, G_{n-1}(x') \,\frac{x_n^k}{k!^2}\,,
$$
where $j$ denotes any integer $\in [1,n-1]$.
\end{remark}

\subsection{Proof of Theorem \ref{theorem:Gnintegral}}\label{subsection:GnIproof}
By taking into account Propositions \ref{proposition:KuD} and \ref{proposition:GnGn-1},  it follows that, for all $x=(x',x_n)$ sufficiently near $0\in\CC^n$,
\begin{equation}\label{Gn.rec.DL}
G_n(x)=\frac 1{2\pi i}{\int_{\cL_R} G_{n-1}(e^\zeta\, x')\, \cC(x_n,\zeta)\,d\zeta}\,,
\end{equation}
where $\cL_R$ denotes a positive-oriented circle $\vert\zeta\vert=R$ with $R$ sufficiently large.
We will express $C(x_n,\zeta)$ as an integral and then prove Theorem \ref{theorem:Gnintegral} by permuting the obtained double integral.

Recall the following Euler's type integral formula for ${}_2F_1$ \cite[p. 58, (6)]{Lu}:
\[
{}_2F_1(a,b;c;u)= \frac{i\Gamma(c) e^{i\pi (a-c)}}{2\Gamma(a)\Gamma(c-a) \sin\pi(c-a)}
\int_{\cC_{(0,1+)}} t^{a-1}(1-t)^{c-a-1}(1-t\,u)^{-b}dt,
\]
where $\Re (b)>0$, $c-a\notin\ZZ_{\ge 0}$, $|\arg(1-u)|<\pi$ and
where the contour $\cC_{(0,1+)}$ goes from zero towards 1 and then comes back to zero after a positive round around 1.
Therefore, in view of \eqref{equation:K2F1}, $\cK(u,z)$ can be rewritten as follows:
\begin{equation}\label{equation:KuzS}
\cK(u,z)= \frac{1}{2\pi i}\int_{\cC_{(0,1+)}} \bigl(\cS(u;t)\bigr)^{z+1}\,\frac{dt}t\,,
\end{equation}
where $\cS(u;t)$ is the transformation  defined in \eqref{equation:Sa} with $a=u$. Notice that the integral contour $\cC_{(0,1+)}$ can be replaced by a loop $\cC_1$ as in Theorem \ref{theorem:GncC}, what we are going to do in the following.

Consequently, the Laplace transform \eqref{equation:Cut} can be decoded as  a double integral. Applying Fubini's Theorem to this one gives raise to the following relation:
$$
\cC(u,\zeta)=\frac 1{2\pi i}\,\int_{\cC_1} \cS(u;t)\,\frac{dt}t\,\int_0^\infty \bigl(\cS(u;t)\,e^{-\zeta}\bigr)^{z}\,dz\,.
$$
Remember that $\cC(u,\zeta)$ represents a germ of analytic function at $(0,\infty)$ in $\CC\times\CC_\infty$, with $\CC_\infty=\CC\cup\{\infty\}$. Let ${\bf D}_R=\{u\in\CC:\vert u\vert<R\}$; for any sufficiently small  $R>0$, if $(u,t)\in {\bf D}_R\times\cC_1$ and if $\cC_1$ remains enough close to $t=1$, then $\cS(u,t)\sim1/(1-t)$, thus one may suppose that $\arg\bigl(e^{i\pi}\,\cS(u,t)\bigr)\in(-\pi,\pi]$. Therefore,  one finds the following expression:
\begin{equation}\label{equation:Cutintegral}
\cC(u,\zeta)=\frac 1{2\pi i}\,\int_{\cC_1} \frac{\cS(u;t)}{\zeta-\log\bigl( \cS(u;t)\bigr)}\,\frac{dt}t\,,
\end{equation}
where $\log$ denotes the principal branch of the logarithm.

Finally, if one puts both $u=x_n$ and the expression \eqref{equation:Cutintegral} in \eqref{Gn.rec.DL} and then applies again Fubini's Theorem to this new double integral, one findsz that
$$
G_n(x',x_n)=-\frac 1{4\pi^2}\,\int_{\cC_1}\cS(x_n;t)\, \frac{dt}t\int_{\cL_R}\frac{G_{n-1}(e^\zeta\,x')}{\zeta-\log\bigl( \cS(x_n;t)\bigr)}\,d\zeta\,.
$$ Thus from the Cauchy's formula one obtains the wanted relation \eqref{equation:Gnintegral} and this is the end of the proof of Theorem \ref{theorem:Gnintegral}.\hfill $\Box$

\begin{remark}\label{remark:GnI}In the same way, one can prove that the function $\cK(x_n,\cD')f$ considered in Proposition \ref{proposition:KuD} can be read as follows:
$$
\cK(x_n,\cD')f(x)=\frac1{2\pi i}\int_{0}^{(1+)}f(S(x_n;t)x')\,\cS(x_n;t)\,\frac{dt}t\,.
$$
\end{remark}

 Notice that one may replace in the relation \eqref{equation:Gnintegral} the integration contour by any smooth closed-loop homotopic to $\partial^+D(1;1/2)$ in $\CC\setminus\{1\}$; such loop will be denoted as $\cC_1$ as in  Theorem \ref{theorem:GncC}.

\subsection{Proof of Theorem \ref{theorem:GncC}}\label{subsection:GnIS}
We proceed to the induction on $n$. If $n=1$, it is elementary to check  that
\begin{equation}\label{equation:G1}
G_1(x)=\frac1{1-x}=\frac 1{2\pi i}\int_{\cC_1}\frac{dt}{(1-t)(tx-1)}\,,
\end{equation}
so that the wanted relation \eqref{equation:GnS} is true for $n=1$.

From the expression \eqref{equation:Gnintegral}, one may write that, if $x=(x', x_n)$ remains enough close to $0$ in $\CC^n$, then:
$$
G_n(x)=\frac 1{2\pi i}\int_{\cC_1}G_{n-1}\bigl(\cS(x_n;t_n)\,x'\bigr)\,\cS(x_n;t_n)\,\frac{dt_n}{t_n}\,.
$$
Assume Theorem \ref{theorem:GncC} to be true for $m=n-1$, $n\ge 2$; it follows that
$$
G_n(x)=\frac {1}{(2\pi i)^n}\,\int_{\cC_1}\cS(x_n;t_n)\,\frac{dt_n}{t_n}\int_{\cC_1^{n-1}}\prod_{j=1}^{n-1}\cS(\bar Y_j(x_n,t_n);\bar t_j)\frac{dt_j}{t_j}\,,
$$
where we set
$$
\bar Y_j(x_n,t_n)=\cS(x_n;t_n)\,\bar x_j,\quad
1\le j\le n-1.
$$
Therefore, one can complete the proof with the help of the definition \eqref{equation:Sam} for the transformations $\cS(\bar x_j;\bar t_j)$.
\hfill$\Box$

\subsection{Analytic continuation domain of $G_n$}\label{subsection:GnIAC}
For all positive integer $n$, let $\Omega_n$ be the analytic continuation domain of $G_n$. We shall make use of Theorem \ref{theorem:Gnintegral} to set out how to describe $\Omega_n$ from $\Omega_{n-1}$. If $n=1$. it is clear that $\Omega_1=\CC\setminus\{1\}$. For all integer $n\ge 2$, it will be convenient to introduce the following notations: to any $a\in\CC$ is associated the set $\cC_{(a)}$ of all smooth closed-loops that are homotopic to $\partial D^+(1;r)$ in $\CC\setminus\{1,1/a\}$, where $r$ denotes any (infinitesimal) real $>0$; for any $\Gamma\in\cC_{(a)}$ we denote by $\cU_{n-1}^{a;\Gamma}$  the set of all $x'\in\CC^{n-1}$ such that $\cS(a;t)x'\in\Omega_{n-1}$ for all $t\in\Gamma$; finally, we define
\begin{equation}\label{equation:cU}
\cU_{n-1}^a=\cup_{\Gamma\in\cC_{(a)}}\cU_{n-1}^{a;\Gamma},\quad
\cU_n=\cup_{a\in\CC}\bigl(\{a\}\times\cU^a_{n-1}\bigr)
\end{equation}
and, their respective connected component containing ${\bf 0}$ will be denoted by $\cU_{n-1;{\bf 0}}^a$ and $\cU_{n;{\bf 0}}$, respectively.

\begin{proposition}\label{proposition:GnIAC}The following assertions hold for all integer $n\ge 2$.
\begin{enumerate}
\item \label{item:GnIACOmega} $\Omega_n=\cU_{n;{\bf 0}}=\cup_{a\in\CC}\cU_{n;{\bf 0}}^{a}$.
\item \label{item:GnIAC0}$\cU_{n-1}^0=\Omega_{n-1}$.
\item \label{item:GnIAC1a}If $a\not=0$, then $\cU_{n-1}^a=a\cU_{n-1}^{1/a}$ and $\cU_{n-1;{\bf 0}}^a=a\cU_{n-1;{\bf 0}}^{1/a}$.
\item \label{item:GnIACK}For any given compact set $K\subset\CC$, it follows that
$$
\cap_{a\in K}\cU_{n-1}^{a}\not=\emptyset\,,\quad
\cap_{a\in K}\cU_{n-1;{\bf 0}}^{a}\not=\emptyset\,.
$$
\end{enumerate}
\end{proposition}

\begin{proof}
The assertions \ref{item:GnIACOmega} and \ref{item:GnIAC0} are evident.

To prove the assertion \ref{item:GnIAC1a}, it suffices to consider the relation $\cS(1/a;t)=a\cS(a;1/t)$.

The last assertion comes from the fact that $\cap_{a\in K}\cC_{(a)}\not=\emptyset$.
\end{proof}

\begin{corollary}\label{corollary:GnIAC}
The intersection set $\Omega_{n}\cap\bigl(\CC^*\bigr)^n$ is left invariant by the group $\GG_n$.
\end{corollary}

\begin{proof}
It suffices to consider Proposition \ref{proposition:GnIAC} and especially its assertions \ref{item:GnIACOmega} and \ref{item:GnIAC1a}.
\end{proof}

\begin{remark}\label{remark:GnIAC}
In the above, the reasoning was made for each domain $\Omega_n$ that was supposed to be in $\CC^n$. However, it is natural to consider some spaces of loops or covering spaces instead of $\CC^n$, that should be studied in a future work. For $n=2$ and $n=3$, see Remarks \ref{remark:G2extension} and \ref{remark:G3Raccati}.
\end{remark}

\section{Symmetries of $G_n$ and $H_n$}\label{section:Sym}

This section is devoted to the study of the invariance of each function $H_n$  under the action of $\GG_n$, as it is announced in Theorem \ref{theorem:Hnsymmetry}. Remember that $H_n=G_n\,\sqrt{Q_n}$, where the factor $Q$ is a symmetric quadratic polynomial; see \eqref{equation:HnQn} and \eqref{equation:TjQn}. Since all the permutations of the coordinates belong to $\GG_n$,  one needs only to consider $G_n$ under the action of anyone of the $n$ transformations $T_{n,j}$ once the index $n$ is given.

\begin{theorem}\label{theorem:GnTn} The following relation holds for any given pair of positive integers $(n,j)$ such that $j\le n$:
\begin{equation}\label{equation:GnTnj}
T_{n,j}\,G_n(x)=- x_j\, G_n(x).
\end{equation}
\end{theorem}

Two different proofs of this statement will be given in the two foolowing paragraphes. In \S\ref{subsection:SymEuler}, we shall consider the alternative representations of $G_n$ given in Remark \ref{remark:GnGn-1}; accordingly, the corresponding identity \eqref{equation:GnTnj} will be interpreted as consequence of some symmetries of the Euler vector fields $\cD'$.

 In \S\ref{subsection:Symintegral}, the starting point is to make use of the integral representation \eqref{equation:Gnintegral}. It will be shown that Theorem \ref{equation:GnTnj} comes from some very simple symmetries of the quadratic transformation $\cS(a;t)$.

 Finally, a few words will be added in \S\ref{subsection:SymHn} for finding, from Theorem \ref{theorem:GnTn}, the invariance of $H_n$ under the action of the group $\GG_n$.

\subsection{Symmetries viewed from Euler vector fields}\label{subsection:SymEuler} Let $T_n:=T_{n,1}$ and recall that $T_n(x)=(T_{n-1}\,x',x_n/x_1)$, where we suppose that $x_1\not=0$. A key point for what we are going to do in this paragraph is  the following relation :
\begin{equation}\label{dt1}
\delta_{x_1}\,(T_{n}f(x))=\delta_{x_1}f(1/x_1,x_2/x_1,\dots,x_{n}/x_1) =-T_{n}(\cD f(x))\,,
\end{equation}
where $f$ denotes any function which is assumed to be holomorphic in a domain $\Omega\subset \CC^{*n} $ such that $T_{n}\Omega=\Omega$.

\begin{proof}[Proof of \eqref{equation:GnTnj} for $j=1$] The identity is obviously true for $n=1$, since $G_1(x)=1/(1-x)$. If $n=2$, we saw in Paragraph \ref{subsection:ALG2} that $G_2(x)=1/\sqrt{Q_2(x)}$, so that the stated result is already established.

We shall prove the general result by induction on $n$. Suppose the identity \eqref{equation:GnTnj} is true for $j=1$ and $m=n-1\ge 1$, where $n$ denotes some {\it fixed} integer $\ge 2$.
For all integer $k\ge 0$, we define
$$
\gamma_k(x')=\bigl(\frac{1}{x_1}\delta_{x_1}^2\bigr)^{k} G_{n-1}(x'),\quad
\gamma_k^*(x')=T_{n-1}\,\gamma_k(x')\,;
$$
By taking into account Remark \ref{remark:GnGn-1},  the expected relation \eqref{equation:GnTnj} for $j=1$ is equivalent to the following relations for $k\in {\mathbb N}$ and $x'\in\CC^{*n-1}$:
\begin{equation}\label{equation:gammak}
\gamma_k^*(x')
=- x_1^{k+1}\,\gamma_k(x')\,.
\end{equation}
This is true for $k=0$, because of the induction hypothesis on $G_{n-1}$; therefore, it remains only to prove \eqref{equation:gammak} for all $k\ge 1$ and we shall do it by induction on $k$ in the following.

Indeed, suppose $ \gamma_{k-1}^*(x')
=- x_1^{k}\,\gamma_{k-1}(x')
$ for some integer $k\ge 1$; it follows that
\begin{eqnarray}\label{equation:delta1k}
\delta_{x_1}\gamma_{k-1}^*(x')=k T_{n-1}\,\gamma_{k-1}(x')- x_1^{k}\delta_{x_1}\gamma_{k-1}(x').
\end{eqnarray}
On the other hand, applying \eqref{dt1} to $T_{n-1}\gamma_{k-1}$ instead of $T_nf$ gives raise to the relation $\delta_{x_1}\gamma_{k-1}^*(x') =-T_{n-1}(\cD' \gamma_{k-1}(x'))$; thus, from \eqref{equation:delta1k} one finds that
\begin{equation}\label{equation:delta1k1}
x_1^{k}\delta_{x_1}\gamma_{k-1}(x')=T_{n-1}((\cD' +k)\gamma_{k-1}(x'))\,.
\end{equation}

If one applies again \eqref{dt1} to the left-hand side of \eqref{equation:delta1k1} with $x'$ instead of $x$, one may find  that
\begin{equation}\label{deltak1}
\delta_{x_1}\left(x_1^{k}\delta_{x_1}\gamma_{k-1}(x')\right)
=-T_{n-1}(\cD'(\cD'+k)\gamma_{k-1}(x'))\,.
\end{equation}
In the same time, combing the relation \eqref{equation:delta1k1} with
\begin{eqnarray*}
\delta_{x_1}\left(x_1^{k}\delta_{x_1}\gamma_{k-1}(x')\right)=k x_1^{k}\delta_{x_1}\gamma_{k-1}(x')+x_1^{k} \delta_{x_1}^2 \gamma_{k-1}(x'),
\end{eqnarray*}
gives raise to the following:
\begin{equation}\label{deltak2}
\delta_{x_1}\left(x_1^{k}\delta_{x_1}\gamma_{k-1}(x')\right)=k T_{n-1}(\c(\cD'+k)\gamma_{k-1}(x'))+x_1^{k+1} \gamma_k(x'),
\end{equation}
where the relation $\delta_{x_1}^2 \gamma_{k-1}(x')=x_1\,\gamma_k(x')$ was used.

Finally, one can deduce from \eqref{deltak1} and \eqref{deltak2} that
\[
x_1^{k+1} \gamma_k(x')=-T_{n-1}(\cD'+k)^2\gamma_{k-1}(x')\,,
\]
which implies the relation \eqref{equation:gammak} with the help of \eqref{Gnk.PDE}. Thus one completes the proof of  \eqref{equation:GnTnj} for $j=1$.
\end{proof}

\subsection{Symmetries of the quadratic transformation $\cS(a;t)$}\label{subsection:Symintegral}
Let $\cS(a;t)$ be the transformation given in \eqref{equation:Sa}; it follows that
\begin{equation}\label{equation:Sadifference}
\cS(a;t)=\frac1{a-1}\,\bigl(\frac1{1-t}-\frac1{1-at}\bigr)\,,
\end{equation}
so that one can easily find the following relations:
\begin{equation}\label{equation:sigmaat}
\cS(a;t)=\cS\bigl(a;\frac1{at}\bigr),\quad \cS(a;t)=\frac1a\,\cS\bigl(\frac1a;\frac1t\bigr)\,.
\end{equation}

\begin{proof}[Proof of \eqref{equation:GnTnj} for $j=n$] Let $x=(x',x_n)\in\CC^n$,  $x_n\not=0$  and define
\begin{equation}\label{equation:gn}
g_n(x,t)=\frac 1{2\pi i}\,G_{n-1}\bigl(\cS(x_n;t)\,x'\bigr)\,\cS(x_n;t);
\end{equation}
then the integral representation \eqref{equation:Gnintegral} can be read as follows:
\begin{equation}\label{equation:Gngn}
G_n(x)=\int_0^{(1+)}g_n(x,t)\,\frac{dt}t\,.
\end{equation}
By taking into account both relations of \eqref{equation:sigmaat}, one deduces respectively that
\begin{equation}\label{equation:gnat}
g_n(x,t)=g_n\bigl(x,\frac1{x_nt}\bigr),\quad
g_n(x,t)=\frac1{x_n}\,g_n\bigl(T_n'\,x,\frac1t\bigr)\,,
\end{equation}
where $T_n'=T_{n,n}$ is given by $T_n'(x)= (x'/x_n,1/x_n)$.

By combining \eqref{equation:Gngn} with \eqref{equation:gnat}, we find finally that
\begin{equation}\label{equation:Gnat}
G_n(x)=-\int_\infty^{(\frac1{x_n}+)}g_n(x,t)\,\frac{dt}t
\end{equation}
and
\begin{equation}\label{equation:Gn1t}
G_n(x)=-\frac1{x_n}\,\int_\infty^{(1+)}g_n(T_n'x,t)\,\frac{dt}{t}\,.
\end{equation}

As $R\to+\infty$, one can easily see that
$$
\int_{\vert t\vert=R}g_n(x,t)\,\frac{dt}t\to0\,,
$$
so that the following relation holds:
$$
\int_\infty^{(1+)}g_n(x,t)\frac{dt}t=-\int_\infty^{(\frac1{x_n}+)}g_n(x,t)\,\frac{dt}t\,.
$$
Thus one can deduce \eqref{equation:GnTnj} from the relations \eqref{equation:Gnat} and \eqref{equation:Gn1t}.
\end{proof}

\begin{remark}\label{remark:Sym}
Such relation as \eqref{equation:GnTnj} exists for all function that can be expressed by an integral of the type $\displaystyle \int f(x,t)\frac{dt}t$ if $f$ be defined as in \eqref{equation:gn} by replacing $G_{n-1}$ by some suitable function.
\end{remark}

\subsection{Invariance of the function $H_n$ by the group $\GG_n$}\label{subsection:SymHn}
Remember that $H_n(x)=\sqrt{Q_n(x)}\,G_n(x)$; see \eqref{equation:HnQn}.
From the relation \eqref{equation:TjQn}, it follows that
$$
T_{n,j}\sqrt{Q_n(x)}=\pm\frac{\sqrt{Q_n(x)}}{x_j}\,.
$$
Moreover, by a standard argument of continuation process, one can see that the symbol ``$\pm$" can be only ``-" , so that one gets $T_{n,j}H_n(x)=H(x)$ from \eqref{equation:GnTnj}.
\hfill$\Box$

\begin{remark}\label{remark:SymHn}
By taking into account Theorem \ref{theorem:G2}, relation \eqref{equation:G1} and the above-established invariance for $H_n$, the proof of Theorem \ref{theorem:Hnsymmetry} will be completed once one will have expressed the function $G_3$ in terms of the Legendre's elliptic function of the first kind.
\end{remark}

\section{Expressions of $G_2$ and $G_3$ deduced from Euler integrals}\label{section:G23}

As it is shown in \eqref{equation:G1},  letting $G_0=1$ in \eqref{equation:Gnintegral} gives rise to  the sum function of the corresponding power series \eqref{equation:Gnseries} with $n=1$.
The main object of this section is devoted to deducing an explicit expression  for each of $G_2$ and $G_3$ by starting from \eqref{equation:Gnintegral}.

Remember that an algebraic expression is found for $G_2$ in \S\ref{subsection:ALG2}, by making use of general transformations of Appell's series. In \S\ref{subsection:G23G2} that follows, this expression will be reformulated by means of the discriminant of a quadratic polynomial.

Putting this algebraic expression of $G_2$ into \eqref{equation:Gnintegral} would allow one to find an {\it explicit} expression for $G_3$: this is what will be shown in \S\ref{subsection:G23G3} and \S\ref{subsection:G23G3L}. In fact, $G_3$ is nearly associated to the Legendre's elliptic function of the first kind; see Theorem \ref{theorem:G3}. This ends finally the proof of Theorem \ref{theorem:Hnsymmetry}.

\subsection{Alternative proof of Theorem \ref{theorem:G2}}\label{subsection:G23G2}
Putting \eqref{equation:G1} in the left hand side of \eqref{equation:Gnintegral} gives raise to the following expression :
$$
G_2(x_1,x_2)=\frac i{2\pi}\int_0^{(1+)}\frac{dt}{(1-t)(1-x_2t)+x_1t}\,.
$$
Write
$$
(1-t)(1-x_2t)+x_1t=x_2(t-t_+)(t-t_-)\,,
$$
where
$$
t_\pm=\frac{1-x_1+x_2\pm\sqrt{(1-x_1+x_2)^2-4x_2}}{2x_2}\,.
$$
If $(x_1,x_2)\to(0,0)$, one observes that
$$
t_+\to\infty\,,\qquad
t_-\to 1,
$$
so that the last integral for $G_2(x)$ may be evaluated  by applying the residues Theorem at $t=t_-$. By observing that
$$
x_2\,\bigl(t_--t_+\bigr)=-\sqrt{(1-x_1+x_2)^2-4x_2}
$$
and
$$
(1-x_1+x_2)^2-4x_2=(1-x_1-x_2)^2-4x_1\,x_2
$$
one obtains the wanted expression \eqref{equation:G2}.
\hfill$\Box$

\begin{remark}\label{remark:G23G2}
By a suitable change of variables, the integral representation formula \eqref{equation:GnS} for $n=2$ can be formally put in the same form as the Aomoto's formula for the Appell's series $F_4$; see \cite[(0.4)]{Yo1} or \cite[p. 124, Example 3.1]{AoK}. This processus remains only formal, since the Aomoto's formula is valid uniquely  when the parameters $\alpha$, $\beta$, $\nu_1$ and $nu_2$ do not belong to $\ZZ$.
\end{remark}

\subsection{Elliptic integral related to $G_3$}\label{subsection:G23G3}
By making use of the expression \eqref{equation:G2}, the integral representation \eqref{equation:Gnintegral} with $n=3$ can be written as follows:
\begin{equation}\label{equation:G30}
G_3(x)=\frac i{2\pi}\int_0^{(1+)}\omega(x;t)\qquad(x=(x_1,x_2,x_3))\,,
\end{equation}
where
\begin{equation}\label{equation:omegaP}
\omega=\frac{dt}{\sqrt{P(x;t)}}
\end{equation}
with
$$
P(x;t)=\bigl((1-t)(1-x_3t)+(x_1+x_2)t\bigr)^2-4x_1x_2t^2\,.
$$
The differential form $\omega(x;t)$ will be considered
over the Riemann surface of $t\mapsto\sqrt{P(x;t)}$, where $x=(x_1,x_2,x_3)$ will be considered as small complex parameter.

By choosing any determination of $\sqrt{x_1x_2}$ in the Riemann surface of $\sqrt z$, we set
$$
a_{\pm}=1-x_1-x_2+x_3\pm2\sqrt{x_1\,x_2}\,,
$$
$$
t_1=\frac{a_+-\sqrt{a_+^2-4\,x_3}}{2\,x_3}\,,\quad
t_2=\frac{a_--\sqrt{a_-^2-4\,x_3}}{2\,x_3}
$$
and
$$
t_3=\frac{a_++\sqrt{a_+^2-4\,x_3}}{2\,x_3}\,,\quad
t_4=\frac{a_-+\sqrt{a_-^2-4\,x_3}}{2\,x_3}\,.\quad
$$
A direct computation allows us to obtain that
\begin{equation}\label{equation:Pt1234}
P(x;t)=x_3^2\,(t-t_1)\,(t-t_2)\,(t-t_3)\,(t-t_4)\,.
\end{equation}

Let $Q=Q_3$ be as in \eqref{equation:HnQn} with $n=3$, that means that
\begin{equation}\label{equation:Q}
Q(x)=\bigl(1-(x_1+x_2+x_3)\bigr)^2-4(x_1\,x_2+x_2\,x_3+x_3\,x_1)\,;
\end{equation}
let
\begin{equation}\label{equation:u}
u(x)=\frac{64\,x_1\,x_2\,x_3}{Q(x)^2}\,.
\end{equation}

\begin{lemma}\label{lemma:aa}
The following relation holds:
\begin{equation}\label{equation:aax3}
(a_+^2-4\,x_3)\,(a_-^2-4\,x_3)=Q(x)^2\bigl(1-u(x)\bigr)\,.
\end{equation}

\end{lemma}

\begin{proof}
From the definition of $a_\pm$, one can find that
\begin{equation}\label{equation:aa}
a_+\,a_-=(1-x_1-x_2+x_3)^2-4\,x_1\,x_2=Q(x)+4\,x_3\,.
\end{equation}
In the same time, it is easy to see that
$$
a_+^2+a_-^2=2\,a_+\,a_-+16\,x_1\,x_2\,,
$$
so that one obtains finally that
$$
(a_+^2-4\,x_3)\,(a_-^2-4\,x_3)=(a_+\,a_--4\,x_3)^2-64\,x_1\,x_2\,x_3\,.
$$
This ends the proof of \eqref{equation:aax3}
\end{proof}

Consider the differential form $\omega$ of \eqref{equation:omegaP}. In order to put $\omega$ into a so-called Riemann's normal form
(see \cite[p. 10]{RL}), we consider  the Mobius transform:
$$
t\ \mapsto\ t'=\frac{(t-t_1)(t_2-t_3)}{(t-t_3)(t_2-t_1)}\,,
$$
which sends $t_1$, $t_2$ and $t_3$ into $0$, $1$ and $\infty$, respectively.

\begin{proposition}\label{proposition:omega} Let $Q$ and $u$ be given in \eqref{equation:Q} and \eqref{equation:u}, respectively. Then the differential form $\omega$ given in \eqref{equation:omegaP} can be put into the following form:
\begin{equation}\label{equation:omegat'}
\omega=\frac{dt'}{\sqrt{\mu\,\tilde P(x;t')}}\,,\quad \tilde P(x;t')=t'\,(1-t')\,(1-\lambda\ t')\,,
\end{equation}
where
\begin{equation}\label{equation:lambda}
\lambda=\frac{1-\sqrt{1-u(x)}}{1+\sqrt{1-u(x)}}
\end{equation}
and
\begin{equation}\label{equation:mu}
\mu=-\frac{Q(x)}{2}\,\bigl(1+\sqrt{1-u(x)}\,\bigr)\,.
\end{equation}
\end{proposition}

 \begin{proof}
 Consider the factorization of the polynomial $P(x;t)$ given in \eqref{equation:Pt1234}. The above-mentioned Mobius transform $t\mapsto t'$ takes $t_4$ to $1/\lambda$ if one denotes by $\lambda$ the cross-ratio of $t_1$, $...$, $t_4$, this means, if
 $$
 \lambda=\frac{(t_4-t_3)(t_2-t_1)}{(t_4-t_1)(t_2-t_3)}\,.$$
 On the other hand, from the definition of $t_1$, ..., $t_4$, it  follows that
 $$
 (t_4-t_1)\,(t_2-t_3)=\frac1{4x_3^2}\,\bigl(-2a_+\,a_-+8\,x_3-2\sqrt{(a_+^2-4\,x_3)(a_-^2-4x_3}\,\bigr)\,;
 $$
 applying Lemma \ref{lemma:aa} and relation \eqref{equation:aa} gives raise to the following relation:
\begin{equation}\label{equation:t41}
 2\,x_3^2\,(t_4-t_1)\,(t_2-t_3)=-Q(x)\,\bigl(1+\sqrt{1-u(x)}\,\bigr)\,.
\end{equation}
By the same way, one can find that
 $$
 2\,x_3^2\,(t_4-t_3)\,(t_2-t_1)=-Q(x)\,\bigl(1-\sqrt{1-u(x)}\,\bigr)\,.
 $$
 Thus one obtains the expected expression \eqref{equation:lambda} for $\lambda$.

In order to obtain the polynomial $\tilde P(x;t')$ of \eqref{equation:omegat'} and the factor $\mu$ given in \eqref{equation:mu}, we will firstly observe that
\begin{equation}\label{equation:mut}
\mu=(t_2-t_3)\,(t_4-t_1)\,x_3^2\,.\end{equation}
Indeed, an elementary calculation permits to notice that the expression \eqref{equation:Pt1234} can be written as follows:
$$
P(x;t)=x_3^2\,\beta\,\frac{t'(1-t')(1-\lambda\,t')}{(t'-\alpha)^4}\,,
$$
where we set
$$
\alpha=\frac{t_2-t_3}{t_2-t_1}\,,\quad
\beta=(t_2-t_3)\,(t_4-t_1)\,(t_3-t_1)^2\,\alpha^2\,.
$$
Moreover, it is easy to see that
$$
dt=\frac{(t_1-t_3)\,\alpha}{(t'-\alpha)^2}\,dt'\,,
$$
so that one can find finally the expression \eqref{equation:omegat'} for which the factor $\mu$ is given by the relation
$$
\mu=\frac{x_3^2\,\beta}{(t_1-t_3)^2\,\alpha^2}\,;
$$
thus we find the expression \eqref{equation:mut}.

Finally, by taking into account the relation \eqref{equation:t41}, we obtain the factor $\mu$ in the form of \eqref{equation:mu} which satisfies the condition \eqref{equation:omegat'}; thus one completes the proof of proposition \ref{proposition:omega}.
\end{proof}

\subsection{Expression of $G_3$ by means of Legendre's elliptic function}\label{subsection:G23G3L} Coming back to Proposition \ref{proposition:omega}, one may find that the following limit behavior holds as $x\to 0\in\CC^3$:
\begin{equation}\label{equation:t1234limit}
t_1,\ t_2\to 1\,;\qquad
t_3,\ t_4\to\infty\,
\end{equation}
and
\begin{equation}\label{equation:lambdamulimit}
\lambda\to 0,\quad
\mu\to-1.
\end{equation}
Accordingly, for any $x$ enough near to $0\in\CC^3$, the contour of the integral \eqref{equation:G30} can be replaced by a closed loop linked $t_1$ and $t_2$. Therefore,  from proposition \ref{proposition:omega} it follows that
\begin{equation}\label{equation:G3t'}
G_3(x)=\frac i{\pi}\int_0^1\frac{dt'}{\sqrt{\mu\,t'\,(1-t')\,(1-\lambda\ t')}}\,,
\end{equation}
where we replaced the closed contour by the interval $(0,1)$ that is taken two times.

In view of \eqref{equation:lambdamulimit}, letting $x=0$ in \eqref{equation:G3t'} leads us to the following identity:
$$
1=\frac i\pi\,\frac1{\sqrt{-1}}\,\int_0^1\frac{dt'}{\sqrt{t'\,(1-t')}}\,,
$$
which suggests us to take $\sqrt{-1}\,=i$; in other words,  we will make use of the following relation, deduced from \eqref{equation:mu}:
\begin{equation}\label{equation:-mu}
\frac i{\sqrt\mu}\,=\frac{\sqrt2}{\sqrt {Q(x)}}\,\bigl(1+\sqrt{1-u(x)}\,\bigr)^{-1/2}
\end{equation}
as well as $x$ remains sufficiently near to $0$ in $\CC^3$.

\begin{theorem}\label{theorem:G3}
Let $Q$ and $u$ be given in \eqref{equation:Q} and \eqref{equation:u}, respectively. For all $x=(x_1,x_2,x_3)$ sufficiently close to $0\in\CC^3$, the following relation holds:
\begin{equation}\label{equation:G3}
G_3(x)=\frac1{\sqrt{Q(x)}}\,{}_2F_1\bigl(\frac14,\frac34;1;u(x)\bigr)\,.
\end{equation}
\end{theorem}

\begin{proof}
By taking into account of \eqref{equation:-mu},  letting $t'=s^2$ allows us to write the integral \eqref{equation:G3t'} as follows:
$$
G_3(x)=\frac{{\sqrt 2}}{\sqrt{Q(x)}}\,\bigl(1+\sqrt{1-u(x)}\,\bigr)^{-1/2}\,\frac{2}{\pi}\,\int_0^1\frac{ds}{\sqrt{(1-s^2)(1-\lambda s^2)}}\,,
$$
where $\lambda$ is given as in \eqref{equation:lambda}. Here, the last integral is an elliptic integral of the first kind; more precisely, from \cite[p. 132, (3.2.3)]{AAR}, one gets the following expression:
$$
G_3(x)=\frac{\sqrt 2}{\sqrt{Q(x)}}\,\bigl(1+\sqrt{1-u(x)}\,\bigr)^{-1/2}\,{}_2F_1\bigl(\frac12,\frac12;1;\frac{1-\sqrt{1-u(x)}}{1+\sqrt{1-u(x)}}\bigr)\,.
$$

Remember the following identity deduced from Pfaff's transformation on $_2F_1$ (see \cite[p. 128, (3.1.9)]{AAR}):
$$
_2F_1(a,b;a-b+1;z)=(1+z)^{-a}\,{}_2F_1\bigl(\frac a2,\frac{1+a}2;a-b+1;\frac{4z}{(1+z)^2}\bigr)\,.
$$
This leads us to the expected relation \eqref{equation:G3}, with
$$
a=b=\frac12,\quad
z=\frac{1-\sqrt{1-u(x)}}{1+\sqrt{1-u(x)}}\,.
$$
Thus one completes the proof of Theorem \ref{theorem:G3}.
\end{proof}

\subsection{End of the proof of Theorem \ref{theorem:Hnsymmetry}}\label{subsection:G23Hn}
It suffices to combine Theorem \ref{theorem:G3} with relation \eqref{equation:G1} and Theorems \ref{theorem:G2} and \ref{theorem:GnTn}; see Remark \ref{remark:SymHn}.
\hfill$\Box$

\section{Point view of ODE on the function $G_3$}\label{section:G3ODE}

The hypergeometric function ${}_2F_1\bigl(\frac14,\frac34;1;u\bigr)$ appeared in the expression \eqref{equation:G3} of $G_3(x)$  can be characterized by certain second order ODE with explicit polynomial coefficients. Or, as it is easy to  se (see Lemma \ref{lemma:G3PDE} below), all member of the family of the symmetric series $G_n$ for $n\ge 1$ can be viewed as the unique solution to some second order PDE Cauchy problem. In the following, we will show how to ``parameterize" the corresponding PDE initial problem for  $G_3(x)$ as being a nonlinear first order ODE problem, that is of Riccati type. This constitutes a new proof of Theorem \ref{theorem:G3}.

In the following,  $(x,y,z)$ will be used instead of $(x_1,x_2,x_3)$.

\subsection{The function $G_3$ regarded as coming from a Cauchy problem}\label{subsection:G3Cauchy}  We start with the following elementary fact.

\begin{lemma}\label{lemma:G3PDE}
The germ of analytic function  $G_3(x)$ at $0\in\CC^3$ satisfies the following Cauchy problem:
\begin{equation}\label{equation:Cauchy}
 \partial_x(x\partial_xF)=\partial_y(y\partial_yF),\quad
F(x,y,0)=G_2(x,y).
\end{equation}
Moreover, it is the unique solution of \eqref{equation:Cauchy} that be analytic at $0$ in $\CC^3$ and symmetric with respect to its variables.
\end{lemma}

\begin{proof}
Direct verification for \eqref{equation:Cauchy}.

For the uniqueness, it suffices to consider the Taylor's series of $F$ at $0\in\CC^3$; the details are left to the reader.
\end{proof}

Remark the above initial condition $F(x,y,0)=G_2(x,y)$ comes from the definition \eqref{equation:Gnseries}. Moreover, it follows that
$$
G_2(x,y)=\frac1{\sqrt{Q_2(x,y)}}=\frac1{\sqrt{Q(x,y,0)}}\,;
$$
see \eqref{equation:G2} and \eqref{equation:Q}.

\subsection{A new parametrization}\label{subsection:G3parametrization}
We will suppose that $F$ is of the following form:
\begin{equation}\label{eq:uvarphi}
F(x,y,z)=\frac{\varphi(u(x,y,z))}{\sqrt {Q(x,y,z)}}\,,
\end{equation}
where $Q$ and $u$ are given as in \eqref{equation:Q} and \eqref{equation:u}, respectively, and where $\varphi$ denotes one unknown function that we would like to determine.
If one makes use of the logarithmic derivation
$$
D_tF:=\frac{\partial_tF}F\quad (t=x,y,z),
$$
then one can write
$$
 \frac{\partial_t(t\partial_tF)}{F}=D_tF+t(D_tF)^2+t\partial_t D_tF.
$$

Let $\psi$ be the logarithmic derivative function of $\varphi$:
$$
\psi(u)=D\varphi(u):=\frac{\varphi'(u)}{\varphi(u)};
$$
one may successively find that, if $t$ denotes one of the variables $x$, $y$ and $z$, then
$$
D_tF=\psi(u) u'_t-\frac12\, D_tQ,
$$
$$
t(D_tF)^2=t(\psi(u))^2{u'_t}^2-t\psi(u)u_t'D_tQ+\frac t4(D_t Q)^2,
$$
$$
t\partial_t(D_tF)=t\psi'(u){u'_t}^2+t\psi(u)u^{''}_{tt}-\frac t2\partial_tD_tQ.
$$
Therefore, if one sets
\begin{equation}\label{eq:A12}
A_{t}=A_{2,t}=t(u_t')^2,\quad A_{1,t}=u_t'-tu_t'D_tQ+tu_{tt}^{''}
\end{equation}
and
\begin{equation}\label{eq:A0}
A_{0,t}=-\frac12D_tQ+\frac t4(D_tQ)^2-\frac t2\partial_tD_tQ,
\end{equation}
one may obtain the following expression:
\begin{equation}\label{equation:F22}
\frac{\partial_t(t\partial_tF)}{F}=A_{t} \psi'(u)+A_{2,t}(\psi(u))^2+A_{1,t}\psi(u)+A_{0,t},
\end{equation}

By taking into account the relations
$$
\partial_tD_tQ=\frac2Q-(D_tQ)^2,\quad
u_t'=(D_tu)u=\bigl(\frac1t-2D_tQ\bigr)\,u
$$
and
$$
u_{tt}^{''}=\bigl(\partial_tD_tu+(D_t u)^2\bigr)\,u=\bigl(-\frac4Q-\frac4tD_tQ+6(D_tQ)^2\bigr)\,u,
$$
the quantities given in \eqref{eq:A12} et \eqref{eq:A0} can be written as follows:
\begin{equation}\label{eq:A1A}
A_{t}=A_{2,t}=\bigl(\frac1t-4D_tQ+4t(D_tQ)^2\bigr)\,u^2,
\end{equation}
\begin{equation}\label{eq:A2A}
A_{1,t}=\bigl(\frac1t-\frac{4t}Q-7D_tQ+8t(D_tQ)^2\bigr)\,u
\end{equation}
and
\begin{equation}\label{eq:A0A}
A_{0,t}=-\frac tQ-\frac12D_tQ+\frac 34t\,(D_tQ)^2.
\end{equation}

\begin{lemma}\label{lemma:Dxf}
The following relations hold:
$$
D_xQ-D_yQ=\frac{4(x-y)}Q
$$
and
$$
x(D_x Q)^2-y(D_y Q)^2=\frac{4(x-y)}Q+\frac{16(x-y)u}{xy}.
$$
\end{lemma}

\begin{proof}
If one introduces the following notations:
$$
X=x+y+z-1,\quad
\hat t=x+y+z-t\quad(t=x,\,y,\,z)\,,
$$
then one can notice that
\begin{equation}\label{equation:DtQ}
D_tQ=\frac{2(X-2\,\hat t)}Q,
\end{equation}
which implies immediately the first relation stated in the above in Lemma \ref{lemma:Dxf}.

Again using \eqref{equation:DtQ} allows us to write
$$
x(D_xQ)^2-y(D_yQ)^2=\frac4{Q^2}\bigl((x-y)X^2-4(x\hat x-y\hat y)X+4(x\hat x^2-y\hat y^2)\bigr)\,,
$$
replace $X^2$, $\hat x$, $\hat y$ by $Q+4(xy+yz+zx)$, $y+z$ and $z+x$, respectively; a direct computation permits to find the second relation of Lemma \ref{lemma:Dxf}, which ends the proof.
\end{proof}

\subsection{Riccati equation deduced from  $G_3$}\label{subsection:G3Riccati}
We are ready to give the principal statement of this section.

\begin{theorem}\label{theorem:riccati}
If $F$ is of the form \eqref{eq:uvarphi}, then the Cauchy problem \eqref{equation:Cauchy} associated to $G_3$ can be transformed as follows:
\begin{equation}\label{eq:riccati}
t(64t-1)\psi'(t)+t(64t-1)(\psi(t))^2+(128t-1)\psi(t)+12=0,
\end{equation}
with
$$
\psi(t)=\frac{\varphi'(t)}{\varphi(t)},\quad
\psi(0)=12.
$$
\end{theorem}

\begin{proof} Consider the expression \eqref{equation:F22} and, for any symbol $B_x$ appearing in the list $A_x$, ..., $A_{0,x}$, define $B_{x|y}=B_x-B_y$. Therefore, the second order PDE of the Cauchy problem \eqref{equation:Cauchy} can be read as follows:
$$
A_{x|y}\, \psi'(u)+A_{2,x|y}\,(\psi(u))^2+A_{1,x|y}\,\psi(u)+A_{0,x|y}=0\,,
$$

On the other hand, applying Lemma \ref{lemma:Dxf} to the relations \eqref{eq:A1A}, \eqref{eq:A2A} and \eqref{eq:A0A} gives raise to the following:
$$
A_{x|y}=A_{2,x|y}=\frac{x-y}{xy}(64u-1) u^2,\quad
A_{1,x|y}=\frac{x-y}{xy}\,(128u-1)\, u
$$
and
$$
A_{0,x|y}=\frac{12(x-y)}{xy}\,u.
$$
Thus, one obtains Eq. \eqref{eq:riccati}.

The initial condition $\psi(0)=12$ is necessary for \eqref{eq:riccati} to have an analytic solution at $t=0$. This is the end of the proof of Theorem \ref{theorem:riccati}.
\end{proof}

Eq. \eqref{eq:riccati} is a Riccati equation, from which one obtains a second order linear equation for $\varphi$ (see \cite[p. 24]{In}):
$$
t(64t-1)\,\varphi''(t)+(128t-1)\,\varphi'(t)+12\,\varphi(t)=0.
$$
Making use of the new variable $s=64t$ allows one to write the last linear equation into an Euler-Gauss hypergeometric equation, so that one can observe that
$$
\varphi(t)={}_2F_1\bigl(\frac14,\frac34;1;64t\bigr).
$$
This constitutes an alternative proof for Theorem \ref{theorem:G3}.

\begin{remark}\label{remark:G3Raccati}Thanks to Theorem \ref{theorem:G3}, one can proceed to the analytic continuation of $G_3$ into a larger domain than  the domain $\Omega_3({\bf 0};1)$ given in \eqref{equation:Omegan} for $n=3$.
For doing this,
it is useful to notice that the following conditions are equivalent:
\begin{enumerate}
\item $u(x,y,z)=1$;
\item $\epsilon_1\,\sqrt x+\epsilon_2\,\sqrt y+\epsilon_3\,\sqrt z=1$, where $\epsilon_j=\pm1$ for $j=1$, $2$ and $3$.
\end{enumerate}

\end{remark}

\bigskip

{\it Acknowledgement--} This work has been started during the period June/July, 2010, while the first author was a visiting professor at USTL (Lille 1) and the most part of the present paper has been outlined when the third author visited in Wuhan university in the period of June/July, 2011. 
This work was partially supported by the NSFC under the grants 11131005, 11171261.

\end{document}